\documentclass[a4paper,USenglish,cleveref,nameinlink,thm-restate,nolineno]{lipics-v2021}

\usepackage{thmtools,thm-restate}

\PassOptionsToPackage{dvipsnames,usenames,cmyk}{xcolor}
\usepackage[utf8]{inputenc}
\usepackage{amsmath, amssymb, amsthm}
\usepackage{xspace}
\usepackage{comment}
\usepackage{todonotes}
\usepackage{graphicx}
\usepackage{enumerate}
\usepackage{microtype}
\microtypesetup{nopatch=item}
\usepackage{subcaption}
\usepackage[nocompress,noadjust]{cite}
\newcommand{\Oh}{{\ensuremath{\mathcal{O}}}}

\hideLIPIcs
\nolinenumbers

\title{
    Transforming Stacks into Queues:\\
    Mixed and Separated Layouts of Graphs
}

\titlerunning{
    Transforming Stacks into Queues: Mixed and Separated Layouts of Graphs
}

\author{Julia~Katheder}{Wilhelm-Schickard-Institut f{\"u}r Informatik, Universit{\"a}t T{\"u}bingen, T{\"u}bingen, Germany}
{julia.katheder@uni-tuebingen.de}
{https://orcid.org/0000-0002-7545-0730}
{The work of J. Katheder is supported by DFG grant Ka 812-18/2.}

\author{Michael~Kaufmann}{Wilhelm-Schickard-Institut f{\"u}r Informatik, Universit{\"a}t T{\"u}bingen, T{\"u}bingen, Germany}
{mk@informatik.uni-tuebingen.de}
{https://orcid.org/0000-0001-9186-3538}
{The work of M. Kaufmann is supported by DFG grant Ka 812-18/2.}

\author{Sergey~Pupyrev}{Menlo Park, CA, USA}
{spupyrev@gmail.com}
{https://orcid.org/0000-0003-4089-673X}{}

\author{Torsten~Ueckerdt}{Institute of Theoretical Informatics, Karlsruhe Institute of Technology, Karlsruhe, Germany}
{torsten.ueckerdt@kit.edu}
{https://orcid.org/0000-0002-0645-9715}
{The work of T. Ueckerdt is supported by DFG - 520723789.}

\authorrunning{J. Katheder, M. Kaufmann, S. Pupyrev, and T. Ueckerdt}
\Copyright{Julia Katheder, Michael Kaufmann, Sergey Pupyrev, and Torsten Ueckerdt}



\theoremstyle{plain}

\newtheorem{open}[theorem]{Open Problem}
\crefname{open}{Open Problem}{Open Problems}

\graphicspath{{figures/}}

\definecolor{defblue}{rgb}{0.121,0.47,0.705}
\newcommand{\df}[1]{{\color{defblue}\it #1}}

\DeclareMathOperator{\sn}{sn}
\DeclareMathOperator{\ssn}{\overline{sn}}

\DeclareMathOperator{\qn}{qn}
\DeclareMathOperator{\sqn}{\overline{qn}}
\DeclareMathOperator{\mn}{mn}
\DeclareMathOperator{\smn}{\overline{mn}}

\begin{document}
\thispagestyle{empty}

\date{}
\maketitle

\begin{abstract}
    Some of the most important open problems for linear layouts of graphs ask for the relation between a graph's queue number and its stack number or mixed number.
    In such, we seek a vertex order and edge partition of $G$ into parts with pairwise non-crossing edges (a stack) or with pairwise non-nesting edges (a queue).
    Allowing only stacks, only queues, or both, the minimum number of required parts is the graph's stack number $\sn(G)$, queue number $\qn(G)$, and mixed number $\mn(G)$, respectively.

    Already in 1992, Heath and Rosenberg asked whether $\qn(G)$ is bounded in terms of $\sn(G)$, that is, whether stacks ``can be transformed into'' queues.
    This is equivalent to bipartite $3$-stack graphs having bounded queue number (Dujmovi\'{c} and Wood, 2005).
    Recently, Alam et al. asked whether $\qn(G)$ is bounded in terms of $\mn(G)$, which we show to also be equivalent to the previous questions.

    We approach the problem by considering \emph{separated} linear layouts of bipartite graphs.
    In this natural setting all vertices of one part must precede all vertices of the other part.
    Separated stack and queue numbers coincide, and for fixed vertex orders, graphs with bounded separated stack/queue number can be characterized and efficiently recognized, whereas the separated mixed layouts are more challenging.
    In this work, we thoroughly investigate the relationship between separated and non-separated, mixed and pure linear layouts.
\end{abstract}

\section{Introduction}
In this paper we study linear layouts of graphs\footnote{Throughout, all graphs in this paper are simple, undirected, finite, and have at least one edge.}.
That is, for a graph $G = (V,E)$, we consider a total order $\sigma$ of its vertex set $V$, while $\sigma$ defines the relative position of its edges.
In particular, we investigate for a graph $G$ its well-known parameters stack number, $\sn(G)$, and its queue number, $\qn(G)$, as well as their common generalization called its mixed number, $\mn(G)$.
Their precise definitions go as follows.

\subparagraph*{Stack, Queue, and Mixed Numbers.}
Given a vertex order $\sigma$, two edges $(u,v)$ and $(x,y)$ in $E$ are said to \df{cross} in $\sigma$ if any order of their endvertices symmetric to $u <_{\sigma} x <_{\sigma} v <_{\sigma} y$ is prescribed by $\sigma$.
Roughly speaking, for the stack number of $G$, we want a vertex order $\sigma$ in which not many edges pairwise cross.
Formally, for an integer $s \geq 1$, an \df{$s$-stack layout} of $G = (V,E)$ is a total order $\sigma$ of $V$ together with a partition of $E$ into subsets $E_1,\dots,E_s$, called \df{stacks}, such that no two edges in the same stack cross.
The minimum number $s$ of stacks needed for a stack layout of graph $G$ is called its \df{stack~number} and denoted by $\sn(G)$.
Stack numbers were first investigated by Bernhart and Kainen~\cite{BK79} in 1979, building on Kainen and Ollman~\cite{K74,O73}\footnote{We remark that $s$-stack layouts are also known as $s$-page book embeddings, where stacks are then called pages. Similarly, the stack number $\sn(G)$ is sometimes called the book thickness or page number of $G$.}

As a concept ``dual'' to $s$-stack layouts, for an integer $q \geq 1$, a \df{$q$-queue layout} of $G = (V,E)$ is a total order $\sigma$ of $V$ together with a partition of $E$ into subsets $E_1,\dots,E_q$, called \df{queues}, such that no two edges in the same queue \df{nest};
that is, there are no edges $(u,v)$ and $(x,y)$ in a queue with $u <_{\sigma} x <_{\sigma} y <_{\sigma} v$ (or any symmetric order).
The minimum number $q$ of queues needed for a queue layout of graph $G$ is called its \df{queue number} and denoted by $\qn(G)$.
Queue numbers were introduced by Heath and Rosenberg~\cite{HR92} in 1992.

Stack and queue layouts are generalized through the notion of a \df{mixed} layout.
For integers $s,q \geq 1$, an \df{$s$-stack $q$-queue layout} of $G = (V,E)$ is a total order $\sigma$ of $V$ together with a partition of $E$ into $s$ stacks and $q$ queues.
The minimum value of $s + q$ needed for an $s$-stack $q$-queue layout of graph $G$ is called its \df{mixed number} and denoted by $\mn(G)$.
Mixed layouts were already considered by Heath and Rosenberg~\cite{HR92} in 1992, while a thorough study started only recently~\cite{P17,MDALAM2022,HMP24}. In contrast to mixed layouts, we call a layout \df{pure} if only stacks or only queues are being used.

\subparagraph*{Comparing Stack, Queue, and Mixed Numbers}
Besides their similar definitions, there are more similarities between \df{$k$-stack graphs}, \df{$k$-queue graphs}, and \df{$k$-mixed graphs}, where a $k$-(stack, queue, mixed) graph is defined as a graph admitting a $k$-(stack, queue, mixed) layout.
For example, in each case we have sparse graphs with only $\Oh(kn)$ edges for $n$ vertices~\cite{BK79,HR92,Pemmaraju92}.
Moreover, stack, queue, and mixed numbers are all bounded for planar graphs~\cite{DJMMUW20,ForsterKMPR23,KaufmannBKPRU20,Yannakakis89} and for bounded treewidth graphs~\cite{ABGKP18,Wie17,DMW05,GanleyH01}.
On the other hand, neither stack nor queue number is bounded for $3$-regular graphs~\cite{Wood08,BaratMW06}.
Already in 1992, Heath, Leighton, and Rosenberg~\cite{HLR92} investigated the relationship between stack and queue layouts; in particular whether stack layouts and queue layouts can be transformed into each other.
They introduced the following fundamental question, which remains unanswered despite a wealth of studies on linear graph layouts.

\begin{open}[Heath, Leighton, and Rosenberg~\cite{HLR92}, see also \cite{DW05,DEHMW21,EHMNSW23}]{\ \\}
	\label{open:rel1}
	Do graphs with bounded stack number have bounded queue number?
\end{open}

We emphasize that a companion question (\emph{``do graphs with bounded queue number have bounded stack number?''})
has been recently resolved in the negative by Dujmovi{\'c} et al.~\cite{DEHMW21}.
One attempt to resolve \cref{open:rel1} was made by Dujmovi{\'c} and Wood~\cite{DW05}
who showed that the problem is equivalent to the question of whether bipartite $3$-stack graphs have bounded queue number.
Note that $1$-stack graphs and $2$-stack graphs are planar, and hence, they have a bounded
queue number~\cite{DJMMUW20}, but the question remains open for $3$-stack graphs.

\begin{figure}[!tb]
	\centering
	\includegraphics{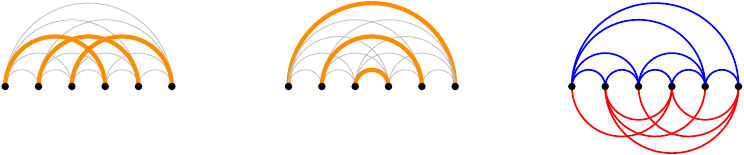}
	\caption{Linear layouts of $K_6$ verifying that $\sn(K_6) \geq 3$ due to a $3$-twist (left), $\qn(K_6) \geq 3$ due to a $3$-rainbow (middle), and $\mn(K_6) \leq 2$ due to a $1$-stack $1$-queue layout (right).}
	\label{fig:K6}
\end{figure}

Turning to mixed layouts, it follows from the definition of the mixed number that for every graph $G$ we have $\mn(G) \le \sn(G)$ and $\mn(G) \le \qn(G)$.
However for some graphs $G$, their mixed number $\mn(G)$ is strictly less than $\sn(G)$ and $\qn(G)$; for example, for the complete graph $K_6$ it holds that $\sn(K_6) = \qn(K_6) = 3$ but $\mn(K_6) = 2$, as illustrated in \cref{fig:K6}.
In particular, graphs with bounded mixed number potentially form a strictly larger class than those of bounded stack number.
Thus, the following problem of Alam et al.~\cite{MDALAM2022} asks for something (potentially) stronger than \cref{open:rel1}.

\begin{open}[Alam et al.~\cite{MDALAM2022}]{\ \\}
	\label{open:rel2}
	Do graphs with bounded mixed number have bounded queue number?
\end{open}

As one of our results (cf.~\cref{thm:equiv}), we shall show that \cref{open:rel1,open:rel2} are in fact equivalent.
That is, either both questions have a \textsc{Yes}-answer or both have a \textsc{No}-answer.
In that sense, it is easier to prove a \textsc{No}-answer by finding graphs of unbounded queue number but bounded mixed number (instead of bounded stack number).
As a second result, also in \cref{thm:equiv}, we shall prove that it indeed suffices to look for graphs of mixed number~$2$; specifically, graphs that admit $1$-stack $1$-queue layouts.
That is, \cref{open:rel1,open:rel2} have a \textsc{Yes}-answer if and only if all $1$-stack $1$-queue graphs have bounded queue number.

\subparagraph*{Separated Linear Layouts}
Since \cref{open:rel1,open:rel2} are likely very challenging in their full generality, we investigate a special case for \df{separated mixed} layouts of bipartite graphs.
A vertex order $\sigma$ of a bipartite graph $G$ with bipartition\footnote{$(A,B)$ is a bipartition if $A \cap B = \emptyset$, $A \cup B = V$, and both induced subgraph $G[A],G[B]$ are edgeless.} $(A,B)$ is \df{separated} if all vertices in $A$ precede all vertices in $B$ (or vice versa).
This naturally gives rise to separated $k$-stack, separated $k$-queue, and separated $s$-stack $q$-queue layouts of bipartite graphs $G$, as well as the corresponding separated stack number $\ssn(G)$, separated queue number $\sqn(G)$, and separated mixed number $\smn(G)$.

Observe that a separated $s$-stack $q$-queue layout can be transformed into a separated $q$-stack $s$-queue layout by reversing the vertex order of one of the parts.
It follows that $\ssn(G) = \sqn(G)$ for every bipartite graph $G$.
Furthermore, every separated vertex order $\sigma$ of $G = (V,E)$ with bipartition $(A,B)$ induces an injective mapping of the edges of $G$ to points of the $|A| \times |B|$ integer grid; see \cref{fig:grid}.
This grid is sometimes referred to as a \df{reduced adjacency matrix} of $G$.
One can easily verify that a subset $S \subseteq E$ of edges forms a queue (resp.~a stack) if and only if the corresponding points form a monotonically increasing (resp.~decreasing) subset in the $|A|\times|B|$ grid.
The stack-queue transformation mentioned above then corresponds to mirroring the grid along the $x$-axis or along the $y$-axis.

\begin{figure}[!tb]
	\begin{subfigure}[b]{.45\linewidth}
		\center
		\includegraphics[page=1]{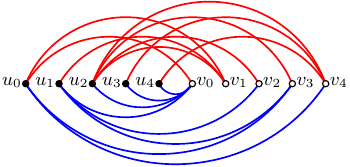}
		\subcaption{separated layout}
	\end{subfigure}
	\hfill
	\begin{subfigure}[b]{.24\linewidth}
		\center
		\includegraphics[page=2]{figures/grid}
		\subcaption{\nolinenumbers{}$2$-layer drawing}
	\end{subfigure}
	\hfill
	\begin{subfigure}[b]{.29\linewidth}
		\center
		\includegraphics[page=3]{figures/grid}
		\subcaption{\nolinenumbers{}grid representation}
	\end{subfigure}
	\caption{Various representations of a separated $1$-stack $1$-queue bipartite graph.}
	\label{fig:grid}
\end{figure}

This separated setting has been widely studied (sometimes, implicitly) under different names,
such as \emph{2-track layouts}~\cite{DPW04,Pupyrev20} or \emph{2-layer drawings}~\cite{Nag05,ALFS20,EW94b}.
For the present paper, let us introduce the following special variant of \cref{open:rel2}.

\begin{open}{\ \\}
	\label{open:rel3}
	Do bipartite graphs with bounded {\bf separated} mixed number have bounded queue number?
\end{open}

\subparagraph*{Our Contributions}

We make progress towards the resolution of \cref{open:rel1,open:rel2}, as well as our new \cref{open:rel3}.
Each of these problems asks whether graphs with low stack number (\cref{open:rel1}), low mixed number (\cref{open:rel2}), or low separated mixed number (\cref{open:rel3}) have also low queue number.
Formally, we say that the \emph{queue number is bounded by stack number} (similarly, \emph{bounded by mixed number}, and \emph{bounded by separated mixed number}) if for every graph $G$ with stack number $\sn(G)$ and queue number $\qn(G)$, it holds that $\qn(G) \le f(\sn(G))$ for some function $f$ (independent of $G$).
The question of whether such functions, $f$, exist are \cref{open:rel1,open:rel2,open:rel3}.

Clearly, a function $f$ that bounds the queue number in terms of the stack number exists if and only if for all $k \in \mathbb{N}$ there is a constant $C = C_k$ such that every $k$-stack graph has queue number at most $C$.
Surprisingly, it is enough to consider just one particular value for $k$.
In fact, Dujmovi\'c and Wood~\cite{DW05} show that \cref{open:rel1} is equivalent to the~existence of a constant $C$ such that every $3$-stack graph has queue number at most~$C$.
(It is known that $1$-stack and $2$-stack graphs have queue number at most $2$~\cite{HLR92} and at most~$42$~\cite{BGR21}, respectively.)

Here, we extend the list to four equivalent statements, showing in particular that \cref{open:rel1,open:rel2} are equivalent, and that it is enough to consider $1$-stack $1$-queue graphs.

\begin{restatable}{theorem}{NonSeparatedEquivalent}
	\label{thm:equiv}
	The following are equivalent:
	\begin{enumerate}[(1)]
		\item \label{equ:nonsep-1} every $s$-stack graph has queue number at most $f(s)$ for some function $f$\\
		(``queue number is bounded by stack number'');

		\item \label{equ:nonsep-2} every $3$-stack graph has queue number at most $C$ for some constant $C$;

		\item \label{equ:nonsep-3} every $1$-stack $1$-queue graph has queue number at most $C$ for some constant $C$;

		\item \label{equ:nonsep-4} every $s$-stack $q$-queue graph has queue number at most $f(s,q)$ for some function $f$\\
		(``queue number is bounded by mixed number'').
	\end{enumerate}
\end{restatable}

\medskip

Turning to separated layouts of (necessarily bipartite) graphs, our main contribution is also a list of four equivalent statements, one of which, namely \eqref{equ:sep-1}, is \cref{open:rel3}.
Statements \eqref{equ:sep-3} and \eqref{equ:sep-4} show that it is enough to consider $1$-stack $q$-queue graphs, respectively even only $1$-stack $6$-queue graphs, for solving \cref{open:rel3}.
Additionally, we discuss a natural approach for \cref{open:rel3} where we start with the grid representation of a separated $s$-stack $q$-queue layout, and then permute the rows and columns so as to obtain a separated $f(s,q)$-queue layout.
Another surprising contribution is that it is always enough to apply the same permutation to the rows and the columns, which is statement \eqref{equ:sep-2}.

\begin{restatable}{theorem}{SeparatedEquivalent}
	\label{thm:same_order}
	The following are equivalent:
	\begin{enumerate}[(1)]
		\item \label{equ:sep-1} every separated $s$-stack $q$-queue graph $G$ has queue number at most $f(s,q)$ for some function~$f$ \qquad (``queue number is bounded by mixed number'');

		\item \label{equ:sep-2} every separated $s$-stack $q$-queue layout of a graph $G=(A \cup B, E)$ with $|A| = |B|$ can be transformed into a separated $f(s,q)$-queue layout for some function $f$ by applying the same permutation to $A$ and $B$;

            \item \label{equ:sep-3} every separated $1$-stack $q$-queue graph has queue number at most $f(q)$ for some function~$f$;

            \item \label{equ:sep-4} every separated $1$-stack $6$-queue graph has queue number at most $C$ for some constant $C$.
	\end{enumerate}
\end{restatable}

\medskip

Finally, let us discuss the connections between non-separated and separated layouts.
Standard techniques for queue layouts (which we present in \cref{sec:preliminaries}) easily give that \cref{open:rel2} implies \cref{open:rel3}.
In fact, \cref{cor:separation} states that for every bipartite graph $G$ we have $\sqn(G) \leq 2\qn(G)$.
However, it remains open whether \cref{open:rel3} implies \cref{open:rel2} (or equivalently \cref{open:rel1}).
The situation is summarized in \cref{fig:overview}.

\begin{figure}[!tb]
    \centering
    \includegraphics[height=7cm]{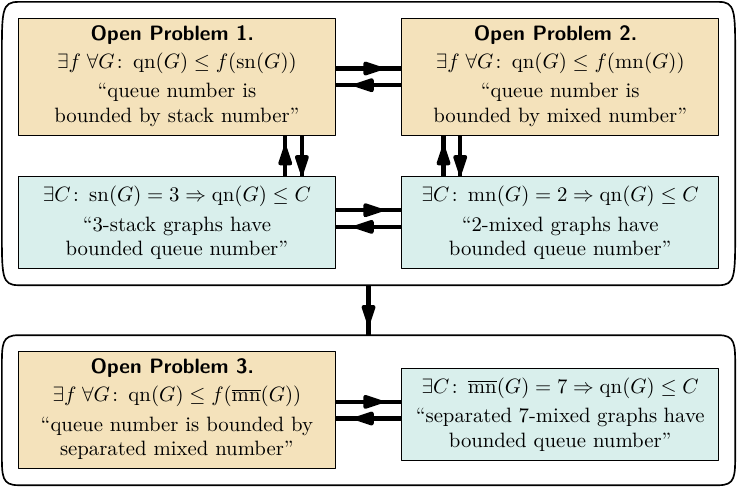}
    \caption{An overview of the new and known relationships between different linear layouts.}
    \label{fig:overview}
\end{figure}

\section{Preliminaries}
\label{sec:preliminaries}

In this section we collect several facts about manipulating vertex orders in linear layouts,
which keep the stack and/or queue numbers within a constant factor from each other.

\begin{lemma}[Riffle-Lemma]\label{lm:321}
	Let $G = (V, E)$ be a graph with a given $q$-queue layout using $\sigma$ as the vertex order and $V_1,\dots, V_k$ be a partition of $V$. Then, for any vertex order $\sigma'$ where for vertices $u,v \in V_i$ it holds that $u <_{\sigma'} v$ whenever $u <_{\sigma} v$:
    \begin{enumerate}[(1)]
        \item $G$ admits a $k^2\cdot q$-queue layout.
        \item\label{item:riffle-bipartite} $G$ admits a $2\cdot \ell \cdot (k - \ell) \cdot q$-queue layout if $G$ is bipartite, $\bigcup_{i=1}^\ell V_i = A$ and $\bigcup_{j=\ell+1}^k V_j = B$.
    \end{enumerate}
\end{lemma}

\begin{proof}
    We show that every queue $Q \subseteq E$ for $\sigma$ can be split into $k^2$ queues for $\sigma'$;
    see \cref{fig:321-1}.
    Overall, this results in the desired $k^2 \cdot q$-queue layout.
    For $i,j \in [k]$ let $E_{i,j} \subseteq Q$ contain all edges $(u,v) \in Q$ such that $u <_{\sigma} v$ and $u \in V_i$ and $v \in V_j$.
    This partitions $Q$ into $k^2$ many edge sets $E_{i,j}$, and we claim that each $E_{i,j}$ is a queue for $\sigma'$.
    For $i=j$, the relative order of any $u,v \in V_i$ is unchanged in $\sigma'$, and hence $E_{i,i}$ is a queue for $\sigma'$.
    For $i\neq j$, let $e_1 =(u_1, v_1)$ and $e_2 =(u_2,v_2)$ be two edges in $E_{i,j}$ with $u_1, u_2 \in V_i$ and $v_1,v_2 \in V_j$.
    Without loss of generality assume that $u_1 <_{\sigma'} u_2$ (hence $u_1 <_{\sigma} u_2$) and assume for a contradiction that $e_1$ nests $e_2$ in $\sigma'$, that is, $u_1 <_{\sigma'} u_2 <_{\sigma'} v_2 <_{\sigma'} v_1$ or $v_2 <_{\sigma'} v_1 <_{\sigma'} u_1 <_{\sigma'} u_2$.
    In either case, it follows that $v_2 <_{\sigma'} v_1$ and hence $v_2 <_{\sigma} v_1$.
    Together with $u_2 <_\sigma v_2$ (as $(u_2,v_2) \in E_{i,j}$) this yields $u_1 <_{\sigma} u_2 <_{\sigma} v_2 <_{\sigma} v_1$ and $e_1$ nests $e_2$ in $\sigma$ -- a contradiction to $Q$ being a queue. It follows that each $E_{i,j}$ requires at most one queue in the layout with vertex order $\sigma'$, which concludes the first case of the proof.

    In the bipartite case, for each queue $Q$ in the original layout, let $e =(u,v)$ be an edge in $Q$ with $u \in V_i$  with $i \leq \ell$ and $v \in V_j$ with $\ell < j \leq k$. 
    As there are $\ell \cdot (k -\ell)$ combinations of $i$ and $j$, and we require two queues, namely $E_{i,j}$ and $E_{j,i}$, for each combination, the resulting layout with vertex order $\sigma'$ requires $2 \cdot \ell \cdot (k -\ell) \cdot q$-queues~in~total.
\end{proof}

Applying the lemma to bipartite graphs with $\ell=1$ and $k=2$ such that $V_1 = A$ and $V_2 = B$ yields the following (well-known) fact; see \cite{Pemmaraju92} for an alternative proof.

\begin{corollary}
	\label{cor:separation}
    Every $q$-queue graph with bipartition $(A,B)$ has a separated $2q$-queue layout.
\end{corollary}

\begin{figure}[!tb]
	\begin{subfigure}[b]{.47\linewidth}
            \center
		\includegraphics[page=1]{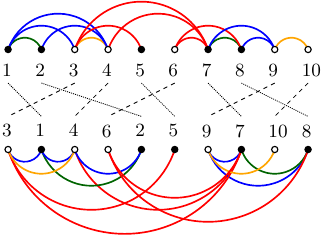}
		\subcaption{}
		\label{fig:321-1}
	\end{subfigure} \hfill
	\begin{subfigure}[b]{.38\linewidth}
            \center
		\includegraphics[page=2]{figures/riffle}
		\subcaption{}
		\label{fig:321-2}
	\end{subfigure}
	\caption{\textbf{(a)} Illustrating \cref{lm:321} (1) and its tightness for $k=2$ \textbf{(b)}. Edge colors illustrate the partition of the edges with respect to $V_1$ (black) and $V_2$ (white). The black lines signify the orders within $V_1$ (dotted) and $V_2$ (dashed).}
	\label{fig:321}
\end{figure}

\medskip

The next two results concern graph subdivisions. A \df{subdivision} of a graph $G$ is a graph obtained from $G$
by replacing each edge $(u, v)$ in $G$ by a path with one or more edges.
Internal vertices on such a path are called \df{division vertices}.

\begin{lemma}[Lemma~2 in \cite{DW05}]
	\label{lm:sub}
	Let $G = (V, E)$ be a graph and $G' = (V', E')$ be a subdivision of $G$ with at most one division vertex per edge.
	If $G$ admits a $q$-queue layout with vertex order $\sigma$, then $G'$ with bipartition $V$ and $V' \setminus V$ admits a separated $(q+1)$-queue layout in which $V$ is ordered as in $\sigma$.
\end{lemma}

A converse operation to a subdivision is a (path-)contraction. We use a more general transformation here.
An \df{r-shallow minor} is a restricted form of a graph minor in which
each (connected) subgraph that is contracted to form the minor has radius at most $r$, where the radius is defined as the minimum of the eccentricities of its vertices. For an $s$-stack $q$-queue graph $G$, we will combine the following result with finding an $s'$-stack $q'$-queue subdivision $H$ where $s'\leq s$ and $q' \leq q$, which has $G$ as an $r$-shallow minor in order to prove equivalence statements in~\cref{thm:equiv} and~\cref{thm:same_order}.

\begin{lemma}[Lemma~12 in \cite{HW21}]
	\label{lm:contraction}
	For every graph $G$ and $H$, where $G$ is a $r$-shallow minor of $H$, it holds that
	$\qn(G) \le (2r + 1) \big(2 \qn(H) \big)^{2r + 1}$.
\end{lemma}

\section{Non-Separated Layouts}
\label{sec:non-separated-layouts}

In this section we explore \cref{open:rel2}.
Dujmovi{\'c} and Wood~\cite{DW05} studied graph subdivisions of stack and queue layouts.
They proved that every $s$-stack graph has
\begin{enumerate}[(i)]
	\item a $3$-stack subdivision with $2 \lceil \log_2 s \rceil - 2$ division vertices per edge, and
	\item a $1$-stack $1$-queue subdivision with $4 \lceil \log_2 s \rceil$ division vertices per edge.
\end{enumerate}
Similarly, every $q$-queue graph has
\begin{enumerate}[(i)]
	\item a $3$-stack subdivision with $2 \lceil \log_2 q \rceil + 1$ division vertices per edge, and
	\item a $2$-queue subdivision with $2 \lceil \log_2 q \rceil + 1$ division vertices per edge, and
	\item a $1$-stack $1$-queue subdivision with $4 \lceil \log_2 q \rceil + 2$ division vertices per edge.
\end{enumerate}

Notice some asymmetry in the statements of the two results:
It is not always possible to sub\-divide an $s$-stack graph with $f(s)$ division
vertices per edge (for an arbitrary function $f$) such that the result admits an $\Oh(1)$-queue layout.
Otherwise, such a subdivision combined with
\cref{lm:contraction} would imply that every $s$-stack graph admits an $\Oh(1)$-queue layout, contradicting the
main result of Dujmovi{\'c} et al.~\cite{DEHMW21}.

\newcommand{\Sh}{{\ensuremath{\mathcal{S}}}}
\newcommand{\Kh}{{\ensuremath{\mathcal{K}}}}
\newcommand{\Qh}{{\ensuremath{\mathcal{Q}}}}

\newcommand{\Vleaf}{{\ensuremath{\widetilde{V}}}}

The crux of the contributions of Dujmovi{\'c} and Wood~\cite{DW05} is a technical construction, which we formalize next.
Let $T$ be a rooted tree with nodes $V(T)$ and edges $E(T)$.
Given a graph $G=(V, E)$, a \df{$T$-partition} of $G$ is a pair $\big(T, \{T_x : x\in V(T)\}\big)$ consisting of a
tree $T$ and a partition of $V$ into sets $\{T_x : x\in V(T)\}$,
such that for every edge $(u, v) \in E$ one of the following holds:
(i)~$u, v \in T_x$ for some $x \in V(T)$, or (ii)~there is an edge $(x, y)$ of $T$ with $u \in T_x$ and $v \in T_y$.
The sets $T_x, x \in V(T)$ are called the \df{bags} of the $T$-partition.
A \df{$T$-layout} of a graph $G$ is a $T$-partition of $G$ in which the bags are ordered, that is,
$<_x$ is a total order of vertices in $T_x$; see~\cref{fig:subdiv_mixed_3_a} for an example. The vertex orders are described in terms of two functions, $\Sh: V(T) \rightarrow \mathbb{N}_0$ and $\Qh: V(T) \rightarrow \mathbb{N}_0$,
defined for the nodes of $T$.
We prescribe each bag to contain either only stacks or only queues formed by intra-bag edges. Here $\Sh(x)$ (respectively, $\Qh(x)$)
denotes the stack (queue) number of the intra-bag edges on $T_x$ under vertex order $<_x$ such that if $T_x$ is prescribed to contain stacks, i.e., if $\Sh(x) > 0$ then $\Qh(x) = 0$, and if $T_x$ is a bag containing queues, then $\Qh(x) > 0$ and $\Sh(x) = 0$.
Similarly, for the edges of $T$, $\Kh(x,y): E(T) \rightarrow \mathbb{N}_0$ is the minimum number of edge sets $E_1,\dots, E_{\Kh(x,y)}$ needed to partition the inter-bag edges of $G$ between $T_x$ and $T_y$ such that they are pairwise non-crossing. Under the concatenation of the
orders $<_x$ and $<_y$, each $E_i$, $1 \leq i \leq \Kh(x,y)$ will form a queue, while each $E_i$ forms a stack when concatenating $<_x$ and $>_y$.
The leaf nodes of a tree $T$ are denoted $\Vleaf(T)$.

In this paper, we work with \df{simple} $T$-layouts, where
\begin{itemize}
	\item for every non-leaf node $x \in V(T) \setminus \Vleaf(T)$, bag $T_x$ forms an independent set in $G$, that is,
	$\Sh(x) = \Qh(x) = 0$;
	\item for every leaf node $x \in \Vleaf(T)$,	bag $T_x$ admits a $1$-stack or a $1$-queue layout under $<_x$, that is,
	$\Sh(x) = 1, \Qh(x) = 0$ or $\Sh(x) = 0, \Qh(x) = 1$;
	\item for every edge $(x, y) \in E(T)$, it holds that $\Kh(x, y) = 1$

\end{itemize}

In all our constructions of tree-partitions,
we utilize a \df{binary tree}, that is, a rooted tree in which every node has at most two children.
The following result provides a subdivision and a tree-layout for a given graph.

\begin{lemma}[a special case of Lemma 21 in~\cite{DW05}]
	\label{lm:t_subdiv}
	Let $G$ be a graph that admits a
	$k$-stack (respectively, $k$-queue) layout with vertex order $\sigma$, and
	$T$ be a binary tree of height $h$ with $k$ or more leaves.
	Then $G$ has a subdivision, $D$, with an
	even number of division vertices per edge such that $D$ has a simple $T$-layout in which
	$\Sh(x) = 1$ (respectively, $\Qh(x) = 1$) for every leaf node $x \in \Vleaf(T)$.
	The number of division vertices per edge is at most $2h$, or exactly $2h$ if all the leaves of $T$ are at depth $h$.
	Moreover, the vertices of $G$ correspond to vertices in the root bag of 
    $T$ and their order in the $T$-layout is $\sigma$, whereas all other bags contain only division vertices.
\end{lemma}

In what follows we consider a (not necessarily proper) $2$-coloring of edges of a tree
using colors blue and red, i.e., we allow edges with the same color at a common endvertex. Throughout, we associate blue colors with stacks, and
red colors with queues. The set of blue edges is $E^s(T) \subseteq E(T)$ and the set of red edges is
$E^q(T)\subseteq E(T)$.

\begin{lemma}[a special case of Lemma 22 in~\cite{DW05}]
	\label{lm:t_layout}
	Let $G$ be a graph with a $T$-layout for some tree $T$.
	Suppose that edges of $T$ are 2-colored in red and blue, and its nodes, $V(T)$, are ordered by $\sigma$ such that the blue edges, $E^s$,
	form a stack and the red edges, $E^q$, form a queue. Let
\begin{align*}
	\lambda_s &= \max_{x \in V(T)} \Big( \Sh(x) + \sum_{(y, x) \in E^s \colon y <_{\sigma} x } \Kh(y,x) + \sum_{(x, y) \in E^s \colon x <_{\sigma} y} \Kh(x, y) \Big),\text{ and} \\
	\lambda_q &= \max_{x \in V(T)} \Big( \Qh(x) + \max_{y \in V(T) \colon y \le_{\sigma} x}  \sum_{(y,z) \in E^q \colon x \le_{\sigma} z} \Kh(y, z)  \Big).
\end{align*}

	Then $G$ admits a mixed $\lambda_s$-stack $\lambda_q$-queue layout.
	Furthermore, the order of vertices in the root bag of the $T$-layout is preserved in the mixed layout.
\end{lemma}

Now we can state the new results of the section.

\begin{lemma}
	\label{lm:subdiv_mixed_3}
	Let $G$ be an $s$-stack $q$-queue graph. Then $G$ has a
        $3$-stack subdivision with \emph{at most} $2 \lceil \log_2 (\max(s,q)) \rceil + 3$ division vertices per edge.
\end{lemma}

\begin{proof}
	First we show how to obtain a tree-layout for an $s$-stack $q$-queue graph $G=(V, E)$ with the vertex order $\sigma$.
	Denote $h = \lceil \log_2(\max(s,q)) \rceil$ so that $\max(s, q) \le 2^h$.
	Assume that $E = S_1 \cup \dots \cup S_s \cup Q_1 \cup \dots \cup Q_q$,
	where $S_i, 1 \le i \le s$ are stacks and $Q_i, 1 \le i \le q$~are~queues.

	Consider the subgraph of $G$ induced by all stack edges, $G^s = (V, S_1 \cup \dots \cup S_s)$.
	Let $T_s$ be a binary tree of height $h + 1$ in which the root node has a single child,
	each internal non-root node has exactly two children, and all leaves are at the same depth.
	By \cref{lm:t_subdiv}, there exists a subdivision
	of $G^s$, denoted $D^s$, with exactly $2(h+1)$ division vertices per edge and a simple $T_s$-layout of $D^s$.
	It holds that $\Sh(x) = 1, \Qh(x) = 0$ for $x \in \Vleaf(T_s)$,
	$\Sh(x) = \Qh(x) = 0$ for $x \in V(T_s) \setminus \Vleaf(T_s)$,
	and $\Kh(x,y) = 1$ for $(x,y) \in E(T_s)$; see \cref{fig:subdiv_mixed_3}.

    \begin{figure}[!tb]
	\begin{subfigure}[b]{.5\linewidth}
		\center
		\includegraphics[page=3,height=5cm]{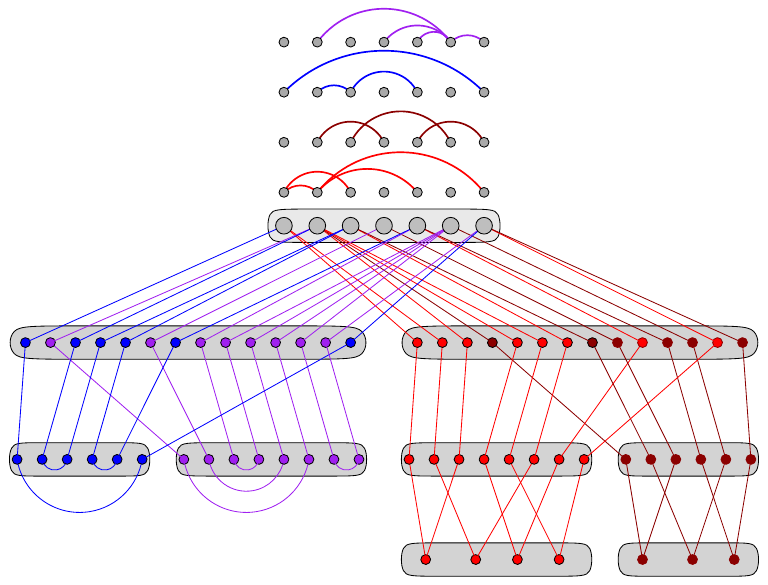}
		\subcaption{A tree-layout of the subdivision of a $2$-stack $2$-queue graph. Edge colors in the tree-layout correspond to the original queues and stacks of $G$ shown above the root node.}
        \label{fig:subdiv_mixed_3_a}
	\end{subfigure}
	\hfill
	\begin{subfigure}[b]{.42\linewidth}
		\center
		\includegraphics[page=4,height=5cm]{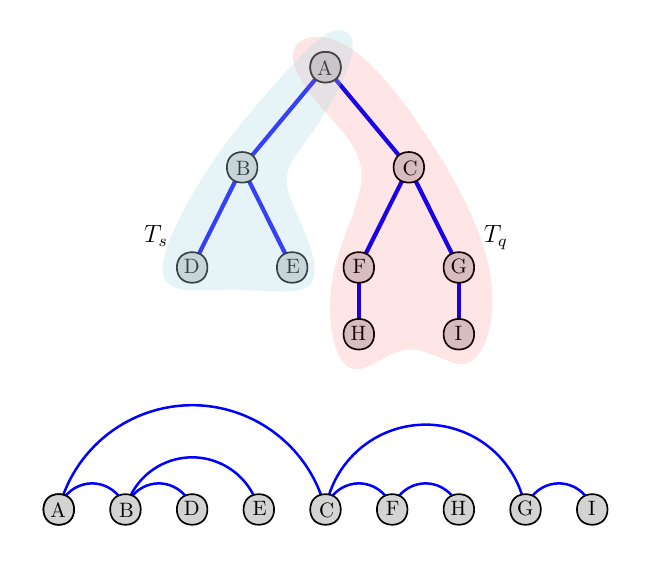}
		\subcaption{\nolinenumbers{}A tree-partition and a $1$-stack layout of the corresponding binary tree, $T_{sq}$, where the blue edge color corresponds to~\cref{lm:t_layout}.}
		\label{fig:subdiv_mixed_3_b}
	\end{subfigure}
	\caption{An illustration for \cref{lm:subdiv_mixed_3}:
		Subdividing a $2$-stack $2$-queue graph $G$ with at most $2 \lceil \log_2(\max(s,q)) \rceil + 3=5$ division vertices per edge
		to obtain a $3$-stack graph.}
	\label{fig:subdiv_mixed_3}
    \end{figure}

	Next consider the subgraph of $G$ induced by the queue edges, $G^q = (V, Q_1 \cup \dots \cup Q_q)$.
	Start with a binary tree of height $h+1$, denoted $T'_q$, which is a copy of $T_s$ (that is,
	the root node has one child and each internal non-root node has two children). There exists a subdivision
	of $G^q$ with exactly $2h+2$ subdivision vertices and its simple $T'_q$-layout by \cref{lm:t_subdiv}. We modify the subdivision and the
	tree-layout as follows. Consider a leaf node $x' \in \Vleaf(T'_q)$ and let $G_{x'}$ be the graph induced by
	vertices in $T_{x'}$ and the corresponding intra-bag edges. Observe that $\Qh(x') = 1$ in the $T'_q$-layout, and hence,
	$\qn(G_{x'}) = 1$ with $<_{x'}$ as the vertex order.
	Now, we apply \cref{lm:sub} by subdividing every edge of $G_{x'}$ once and let $G_x$ be the resulting graph. This results in a $2$-queue layout of the subdivision $G_x$ in which the vertices of $G_{x'}$ are separated from
	the new division vertices and are ordered as in the $T'_q$-layout. Thus, we can build a new tree, denoted
	$T_q$, by appending a child, $x$, to every node $x' \in \Vleaf(T'_q)$. This results in a (non-simple)
	$T_q$-layout for $G_q$, where the leaf nodes of $T_q$ contain the division vertices in the order given by \cref{lm:sub}.
	Since the queue number of $G_{x}$ is at most $2$, the inter-bag edges between $G_{x'}$ and $G_{x}$ can be partitioned into two sets of pairwise non-crossing edges, i.e., we have that $\Kh(x', x) = 2$ for all $(x', x) \in E(T_q)$, where $x \in \Vleaf(T_q)$; see~\cref{fig:subdiv_mixed_3_a} for an illustration.

	By \cref{lm:t_subdiv}, the root bags of the $T_s$-layout and the $T_q$-layout contain the same vertices, $V$, both ordered by $\sigma$.
	Thus, we can merge the two tree-layouts into a joint one, called $T_{sq}$-layout; see \cref{fig:subdiv_mixed_3_b}.
	The corresponding subdivision of $G$, denoted $D^{sq}$, has stack edges subdivided $2h+2$ times and
	queue edges subdivided $2h+3$ times. To apply \cref{lm:t_layout} for the $T_{sq}$-layout, we color all the edges
	of $T_{sq}$ blue and find a $1$-stack layout of the tree via a depth-first search traversal
	such that every node precedes its children in the order.
	Let us argue that the result of applying \cref{lm:t_layout} is a $3$-stack layout, that is,
	$\lambda_s = 3$ and $\lambda_q = 0$. 

	Note that $\Qh(x) = 0$ for all nodes of $T_{sq}$ and the tree contains no red edges; thus, $\lambda_q = 0$.
	For the stack number, consider four disjoint groups of nodes of $T_{sq}$:

	\begin{itemize}
		\item for $x \in \Vleaf(T_s)$, it holds that $\Sh(x) = 1$ and $\Kh(y,x)=1$ for $(y,x) \in E^s$;

		\item for $x \in V(T_s) \setminus \Vleaf(T_s)$ and $x \in V(T'_q) \setminus \Vleaf(T'_q)$,
		it holds that $\Sh(x) = 0$, there exists one edge $(y,x) \in E^s$ with $\Kh(y,x)=1$ and
		two edges $(x,y) \in E^s$ with $\Kh(x,y)=1$;

		\item for $x \in \Vleaf(T'_q)$, it holds that
		$\Sh(x) = 0$, $\Kh(y,x)=1$ for a single edge $(y,x) \in E^s$, and $\Kh(x,y)=2$ for a single edge $(x,y) \in E^s$;

		\item for $x \in \Vleaf(T_q)$, it holds that
		$\Sh(x) = 0$, and $\Kh(y,x)=2$ for a single edge $(y,x) \in E^s$.
	\end{itemize}

	Combining everything together, we get
	\begin{equation*}
		\begin{aligned}
			\lambda_s & = \max_{x \in V(T_{sq})} \Big( \Sh(x) + \sum_{(y, x) \in E^s} \Kh(y,x) + \sum_{(x, y) \in E^s} \Kh(x, y) \Big)
			& \le 3.
		\end{aligned}
	\end{equation*}

	Therefore, subdivision $D^{sq}$ of $G$ admits a $3$-stack layout by \cref{lm:t_layout}.
\end{proof}

With a similar technique we can find a $1$-stack $1$-queue subdivision (\cref{lm:subdiv_mixed_11}) instead of a $3$-stack subdivision (\cref{lm:subdiv_mixed_3}).

\begin{lemma}
	\label{lm:subdiv_mixed_11}
	Let $G$ be an $s$-stack $q$-queue graph. Then $G$ has a
        $1$-stack $1$-queue subdivision with \emph{at most} $4 \lceil \log_2(\max(s,q)) \rceil + 6$ division vertices per edge.
\end{lemma}
\begin{proof}
	First we show how to obtain a tree-layout for a given $s$-stack $q$-queue graph $G=(V, E)$;
	assume that $\sigma$ is the vertex order of $G$.
	Denote $h = \lceil \log_2(\max(s,q)) \rceil$ so that $\max(s, q) \le 2^h$.
	Assume that the edges of $G$ consist of $E = E^s_1 \cup \dots \cup E^s_s \cup E^q_1 \cup \dots \cup E^q_q$,
	where $E^s_i$ is the $i$-th stack and $E^q_i$ is the $i$-th queue; see \cref{fig:subdiv_mixed22ii}.

    Consider a subgraph of $G$ induced by the stack edges, $G^s = (V, E^s_1 \cup \dots \cup E^s_s)$.
	Start with a binary tree, $T'_s$, of height $h + 1$ in which the root node has a single child,
	each internal non-root node has exactly two children, and all leaves are at the same level.
	Build a binary tree $T_s$ from $T'_s$ by subdividing
	every edge once; the height of $T_s$ is $2(h+1)$. By \cref{lm:t_subdiv}, there exists a subdivision of $G^s$,
	denoted $D^s$, with exactly $4(h+1)$ division vertices per edge and a simple $T_s$-layout of $D^s$.
	For the tree-layout, it holds that $\Sh(x) = 1, \Qh(x) = 0$ for the leaf nodes $x \in \Vleaf(T_s)$, and
	$\Kh(x,y) = 1, (x,y) \in E(T_s)$. Furthermore, the root bag of $T_s$-layout consists of vertices $V$, which are ordered by $\sigma$.

	Next consider the subgraph of $G$ induced by the queue edges, $G^q = (V, E^q_1 \cup \dots \cup E^q_q)$.
	To construct a tree-layout of $G^q$, we build a binary tree $T_q$ by copying $T_s$ and appending a node to
	every leaf; the resulting tree has height $2(h+1)+1$.
	By \cref{lm:t_subdiv}, there exists a subdivision
	of $G^q$, denoted $D^q$, with exactly $4(h+1)+2$ division vertices per edge and a simple $T_q$-layout of $D^q$.
	For the layout, we have $\Sh(x) = 0, \Qh(x) = 1$ for the leaf nodes $x \in \Vleaf(T_q)$,
	and $\Kh(x,y) = 1$ for $(x,y) \in E(T_q)$. Again, the root bag of $T_q$-layout contains vertices $V$ ordered by $\sigma$.

	Since the root bags of $T_s$-layout and $T_q$-layout share the same vertex set ordered by $\sigma$,
	we can merge them into a simple $T_{sq}$-layout by identifying the root bags.
	In the corresponding subdivision of $G$, denoted $D^{sq}$,
	the stack edges are subdivided $4h+4$ times, while the queue edges are subdivided $4h+6$ times.
	Now, we color the edges of $T_{sq}$ as follows. For every node with two children, color one of the outgoing edges blue and the other red.
	For every node, $x \in V(T_{sq})$, with a single child, color the edge red if $x$ is at depth $\le 2h$ and color it blue
	if $x$ is at depth $2h+1$;	refer to \cref{fig:subdiv_mixed22ii_b}.
	Observe that the blue edges form a matching, while the red edges are a collection of paths in the tree.

	In order to build a linear layout of $D^{sq}$ using \cref{lm:t_layout}, we order the nodes of $T_{sq}$ so that the blue
	edges form a stack and the red edges form a queue. To this end, order the nodes of $T_{sq}$ level-by-level
	starting from the root. Denote the set of nodes at level $i$ by $L_i$ and suppose that we constructed an
	order $<_{i}$ for $L_i$; the task is to build the order $<_{i+1}$ for $L_{i+1}$. Every node in $L_{i+1}$ has
	a unique parent in $L_i$ adjacent via a blue or a red edge. Sort the nodes in $L_{i+1}$ such that
	we first have all nodes with a blue parent edge ordered by $>_{i}$, and then we have all nodes with a red parent edge ordered by $<_{i}$.
	It is easy to see that the blue edges (forming a matching) do not cross in the layout, while the red edges do not nest.
	Let us denote the final resulting node order of $V(T_{sq})$ by $\sigma$.

	\begin{figure}[!tb]
		\begin{subfigure}[b]{.56\linewidth}
			\center
			\includegraphics[page=1,height=6cm]{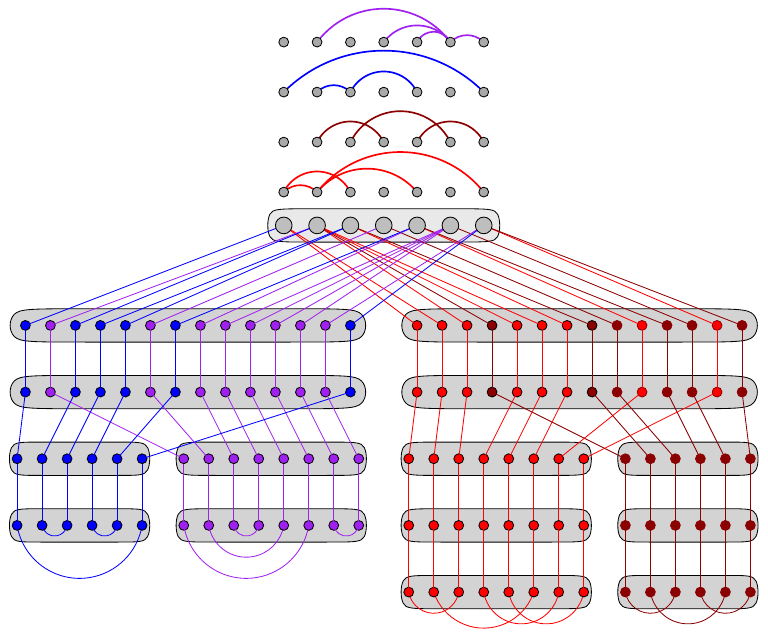}
			\subcaption{A tree-layout of the subdivision of a $1$-stack $1$-queue graph}
		\end{subfigure}
		\hfill
		\begin{subfigure}[b]{.42\linewidth}
			\center
			\includegraphics[page=2,height=6cm]{figures/subdivisions_new}
			\subcaption{\nolinenumbers{}A tree-partition and a $1$-stack $1$-queue layout of the corresponding binary tree, $T_{sq}$}
			\label{fig:subdiv_mixed22ii_b}
		\end{subfigure}
		\caption{An illustration for \cref{lm:subdiv_mixed_11}:
			Subdividing a $2$-stack $2$-queue graph
			with at most $4 \lceil \log_2 \max(s, q) \rceil + 6=10$ division vertices per edge
			to obtain a mixed $1$-stack $1$-queue graph.}
		\label{fig:subdiv_mixed22ii}
	\end{figure}

	Now we apply \cref{lm:t_layout} to $T_{sq}$, and shall argue that the result is a $1$-stack $1$-queue layout of $D^{sq}$.
	First we show that $\lambda_s \le 1$. Consider the set, $\Vleaf(T_s)$, of the leaf nodes of $T_s$, which is a subtree of $T_{sq}$.
	We have $\Sh(x) = 1$ for $x \in \Vleaf(T_s)$ and $\Sh(x) = 0$ for all other nodes, $x \in V(T_{sq}) \setminus \Vleaf(T_s)$.
	By construction, the blue edges of $T_{sq}$ form a matching and are not incident to nodes $\Vleaf(T_s)$.
	Thus,
	\begin{equation*}
		\begin{aligned}
			\lambda_s & = \max_{x \in V(T_{sq})} \Big( \Sh(x) + \sum_{(y, x) \in E^s} \Kh(y,x) + \sum_{(x, y) \in E^s} \Kh(x, y) \Big) \\
			& \le \max \Big\{ \max_{x \in \Vleaf(T_s)} \big( \Sh(x) \big),
			\max_{x \in V(T_{sq}) \setminus \Vleaf(T_s)} \big( \sum_{(y, x) \in E^s} \Kh(y,x) + \sum_{(x, y) \in E^s} \Kh(x, y) \big) \Big\} \\
			& \le \max \Big\{ 1,
			\max_{x \in V(T_{sq}) \setminus \Vleaf(T_s)} \sum_{(y, x) \in E^s} \Kh(y,x) \Big\} = 1.
		\end{aligned}
	\end{equation*}

	Now we show that $\lambda_q \le 1$. By construction, $\Qh(x) = 1$ for $x \in \Vleaf(T_q)$ and $\Qh(x) = 0$ for $x \in V(T_{sq}) \setminus \Vleaf(T_q)$. The red edges of $T_{sq}$ (that is, $E^q(T_{sq})$) are not incident to the leaves of $T_q$, which implies that
	the set $\{ (y,z) \in E^q : x \le_{\sigma} z \}$ is empty for $x \in \Vleaf(T_q)$.
	Every non-leaf node of $T_{sq}$ has at most one outgoing incident red edge in the layout of $T_{sq}$, that is,
	the set $\{ (y,z) \in E^q \mid y \le_{\sigma} x \le_{\sigma} z \}$ contains at most one edge.
	Hence,
	\begin{equation*}
		\begin{aligned}
			\lambda_q & = \max_{x \in V(T_{sq})} \Big( \Qh(x) +  \max_{y \in V(T_{sq}) \colon y \le_{\sigma} x}  \sum_{(y,z) \in E^q \colon x \le_{\sigma} z} \Kh(y, z)  \Big) \\
			& \le \max \Big\{
			\max_{x \in \Vleaf(T_q)} \big( \Qh(x) \big),
			\max_{x \in V(T_{sq}) \setminus \Vleaf(T_q)} \Big( \max_{y \in V(T_{sq}) \colon y \le_{\sigma} x}  \sum_{(y,z) \in E^q \colon x \le_{\sigma} z} \Kh(y, z) \Big)
			\Big\} = 1.
		\end{aligned}
	\end{equation*}

	Therefore, subdivision $D^{sq}$ of $G$ admits a $1$-stack $1$-queue layout by \cref{lm:t_layout}.
\end{proof}

We now prove \cref{thm:equiv}, in particular that \cref{open:rel1,open:rel2} are equivalent.

\NonSeparatedEquivalent*

\begin{proof}
	It is immediate that \eqref{equ:nonsep-4} implies \eqref{equ:nonsep-1}, \eqref{equ:nonsep-2}, and \eqref{equ:nonsep-3}.
        That \eqref{equ:nonsep-1}$\iff$\eqref{equ:nonsep-2} is proven by Dujmovi{\'c} and Wood~\cite{DW05}.

	Now we show that \eqref{equ:nonsep-2}~$\Rightarrow$~\eqref{equ:nonsep-4}.
	Assume that every $3$-stack graph has a bounded queue number, $C \in \Oh(1)$.
	Let $G$ be an $s$-stack $q$-queue graph. By \cref{lm:subdiv_mixed_3}, $G$ admits a $3$-stack
	subdivision, $D$, with at most $k = 2 \lceil \log_2(\max(s,q)) \rceil + 3$ division vertices per edge.
	By the assumption, $\qn(D) \le C$. Note that $G$ is an $r$-shallow minor of $D$ for
	$r = (k+1)/2 = \lceil \log_2 (\max(s,q)) \rceil + 2$.
	By \cref{lm:contraction}, we get
	\[
		\qn(G) \le 	(2r + 1) \big(2 \qn(D) \big)^{2r + 1}
		= (2 \lceil \log_2(\max(s,q)) \rceil + 5) \big(2 C \big)^{2 \lceil \log_2(\max(s,q)) \rceil + 5},
	\]
	which proves that $G$ has a bounded queue number, as long as $s,q$ and $C$ are constants.

	Similarly we show that \eqref{equ:nonsep-3}~$\Rightarrow$~\eqref{equ:nonsep-4}.
	Suppose that every $1$-stack $1$-queue graph has queue number at most $C \in \Oh(1)$.
	Let $G$ be an $s$-stack $q$-queue graph. By \cref{lm:subdiv_mixed_11}, $G$ admits a $1$-stack $1$-queue
	subdivision, $D$, with $k = 4 \lceil \log_2 (\max(s, q)) \rceil + 6$ division vertices per edge.
	By the assumption, $\qn(D) \le C$, and $G$ is an $r$-shallow minor of $D$ for
	$r = \lceil (k+1)/2 \rceil$. Again by \cref{lm:contraction}, we get
	\[
		\qn(G) \le 	(2r + 1) \big(2 \qn(D) \big)^{2r + 1}
		= (4 \lceil \log_2 (\max(s,q)) \rceil + 8) \big(2 C \big)^{4 \lceil \log_2 (\max(s,q)) \rceil + 8},
	\]
	which shows that $G$ has a bounded queue number, as long as $s,q$ and $C$ are constants.
\end{proof}

\section{Separated Layouts}
\label{sec:separated-layouts}

In this section we investigate separated mixed layouts of bipartite graphs and \cref{open:rel3}.
Recall that a vertex order $\sigma$ of a bipartite graph $G = (V, E)$ with bipartition $(A,B)$ is \df{separated}
if all vertices of $A$ precede all vertices of $B$ in $\sigma$ (or vice versa).

\begin{observation}
	\label{obs:separated-stack-is-queue}
	Every separated $s$-stack $q$-queue layout of a bipartite graph $G = (V,E)$ with bipartition $(A,B)$ can be transformed to a separated $q$-stack $s$-queue layout by reversing the order of all vertices in $A$ (or alternatively all vertices in $B$).
\end{observation}

Reversing the order of some consecutive vertices (as in \cref{obs:separated-stack-is-queue}) is the simplest modification we could do to a given layout.
It turns out that this is already enough to show that graphs with separated $1$-stack $1$-queue layouts have a constant queue number.

\begin{theorem}
	\label{thm:1s1q}
	Every separated $1$-stack $1$-queue graph admits a separated $4$-queue layout.
\end{theorem}

\begin{proof}
    Let $G = (V,E)$ be a bipartite graph with bipartition $(A,B)$ admitting a separated $1$-stack $1$-queue layout with vertex order $\sigma$.
    Consider the reduced adjacency matrix $M$ with columns $A = (a_1,\ldots,a_m)$ and rows $B = (b_1,\ldots,b_n)$ ordered according to $\sigma$.
    The edges $S \subset E$ in the stack form a monotonically weakly decreasing subset of $|A| \times |B|$.
    The edges $Q \subset E$ in the queue form a monotonically weakly increasing subset of $|A| \times |B|$.
    Without loss of generality we can assume (by adding edges to the graph) that $S$ and $Q$ correspond to inclusion-maximal such subsets;
    see \cref{fig:11-separated}.

    \begin{figure}[!b]
        \centering
        \includegraphics{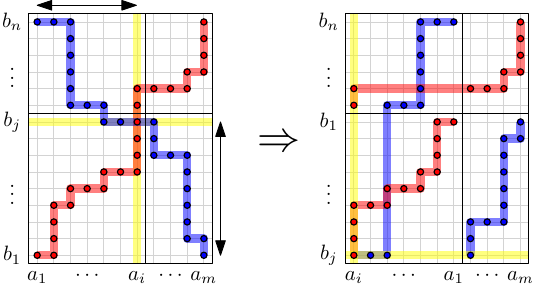}
        \caption{Obtaining a separated $4$-queue layout from a separated $1$-stack $1$-queue layout by reversing column and row orders. The blue edge color of the transformed stack edges is preserved for clarity.}
        \label{fig:11-separated}
    \end{figure}

    Let $(a_i,b_j) \in A \times B$ be a point/edge that is contained in both, $S$ and $Q$.
    Now consider the vertex order $a_i,\ldots,a_1,a_{i+1},\ldots,a_m$ of the columns obtained by reversing the order of $a_1,\ldots,a_i$.
    Similarly, consider the vertex order $b_j,\ldots,b_1,b_{j+1},\ldots,b_n$ of the rows obtained by reversing the order of $b_1,\ldots,b_j$.
    Let $\sigma'$ denote the resulting separated vertex order of $G$.
    Now the edges in $Q$ incident to $a_1,\ldots,a_i$ form a queue $Q_1$ in $\sigma'$.
    Also the remaining edges in $Q$, incident to $a_{i+1},\ldots,a_m$ form a queue $Q_2$ in $\sigma'$.
    Similarly, the edges in $S$ incident to $b_1,\ldots,b_j$ form a queue $Q_3$ in $\sigma'$.
    Also the remaining edges in $S$, incident to $b_{j+1},\ldots,b_m$ form a queue $Q_4$ in $\sigma'$;
    see again \cref{fig:11-separated}.
    This is a separated $4$-queue layout of $G$. 
\end{proof}

Note that by applying~\cref{obs:separated-stack-is-queue}, separated $1$-stack $1$-queue graphs also admit a separated $4$-stack layout. However, we do not know whether the bound of $4$ in \cref{thm:1s1q} is best-possible and remark that $K_{3,3}$ admits a separated $1$-stack $1$-queue layout but its separated stack/queue number is $3$.

\medskip

The simple operation of reversing the order of some consecutive vertices can be employed in a ``checkerboard'' fashion.
Consider a given separated mixed layout (e.g., with $s$ stacks and $q$ queues).
Assume we have partitioned the reduced adjacency matrix by means of a superimposed grid with a checkerboard odd-even pattern, in such a way that odd grid cells contain no stack edges and even grid cells contain no queue edges.
Then, the vertex order of every second row and column can be reversed to derive a separated pure queue layout.
Finding such a grid refinement is always possible, that is, down to individual rows and columns, however the derived $\ssn$ and $\sqn$ are dependent on the grid granularity and thus may not be bounded by $\smn$.
An example for specific separated $s$-stack $q$-queue layouts with bounded $\sqn$ number are grids where each cell contains either an increasing or decreasing diagonal.
In this case, grid columns and rows can be halved in order to apply the checkerboard approach; see \cref{fig:cb-refinement}.

\begin{figure}[!tb]
    \centering
	\begin{subfigure}[b]{.23\linewidth}
		\center
		\includegraphics[page=1]{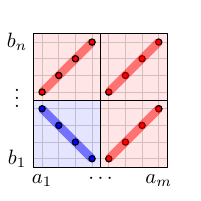}
		\subcaption{}
            \label{fig:cb-refinement-a}
	\end{subfigure}
	\hfill
	\begin{subfigure}[b]{.23\linewidth}
		\center
		\includegraphics[page=2]{figures/checkerboard}
		\subcaption{}
            \label{fig:cb-refinement-b}
	\end{subfigure}
        \hfill
        \begin{subfigure}[b]{.23\linewidth}
		\center
		\includegraphics[page=3]{figures/checkerboard}
		\subcaption{}
            \label{fig:cb-refinement-c}
	\end{subfigure}
        \hfill
        \begin{subfigure}[b]{.23\linewidth}
		\center
		\includegraphics[page=4]{figures/checkerboard}
		\subcaption{}
            \label{fig:cb-refinement-d}
	\end{subfigure}
	\caption{Transforming a separated mixed layout of a grid with diagonals \textbf{(a)} into a separated pure layout \textbf{(d)} by halving columns and rows \textbf{(b)} and reversing every second row and column \textbf{(c)}. The blue edge color of the transformed stack edges is preserved for clarity.}
	\label{fig:cb-refinement}
\end{figure}

However there exist graphs, where such operations do not suffice and more involved row and column permutations are necessary.


In the following, we provide a challenging graph that admits a separated $1$-stack $2$-queue layout and a fairly simple $4$-queue pure layout. However, the grid granularity in the ``checkerboard'' strategy has to be super-constant in order to transform the mixed into a pure layout.

Let us consider, for $n = 2^k$ being any power of~$2$, the following subcubic bipartite graph $G_n = (V \cup U, E)$ that consists of two parts $V = \{v_0,\ldots,v_{n-1}\}$, $U = \{u_0,\ldots, u_{n-1}\}$ and three types of edges:

\begin{itemize}
	\item \emph{red} edges: \qquad $(v_i, u_i)$ for $i = 0,\ldots,n-1$
	\item \emph{brown} edges: \qquad $(v_i, u_{2i})$, $(v_i, u_{2i+1})$ for $i = 0,\ldots,\frac{n}{2}-1$
	\item \emph{blue} edges: \qquad $(v_i, u_{2n-2i-2})$, $(v_i, u_{2n-2i-1})$ for $i = \frac{n}{2},\ldots,n-1$
\end{itemize}

Each $G_n$ admits a separated $1$-stack $2$-queue layout by simply taking the vertex order $v_0,\ldots,v_{n-1},u_0,\ldots,u_{n-1}$, in which red and brown edges for one queue each, while blue edges form one stack; see \cref{fig:g2-grid}.
Note how simply structured these graphs are, and how close this separated mixed layout is to a ``nice grid with diagonals''.
And in fact, there is also a nice separated $4$-queue pure layout as shown in \cref{fig:graph2_queues}.
However, transforming the mixed into a pure layout is non-trivial.

Let us describe the column and row permutation that does the transformation.
These involve periodically reversing and displacing consecutive rows/columns in incrementally increasing blocks.
In particular, consider these four permutation functions $f_1,f_2,f_3,f_4$: 
\begin{align*}
	f_1(x_1,x_2,x_3,x_4) &= (x_1,x_2,x_3,x_4), & f_2(x_1,x_2,x_3,x_4) &= (x_2,x_1,x_4,x_3), \\
	f_3(x_1,x_2,x_3,x_4) &= (x_3,x_4,x_1,x_2), & f_4(x_1,x_2,x_3,x_4) &= (x_4,x_3,x_2,x_1)
\end{align*}
For the graph $G_n$ of size $n = 2^k$, the rows and columns in the initial mixed layout are divided into blocks of consecutive four.
For the $i$-th block, $i = 1,\ldots,\frac{n}{4}$, the permutation function $f_j$ with $j = i \bmod 4$ is applied.
In future steps, each block is treated as a single entity, keeping its internal order intact.
Then again, the rows and columns are divided into blocks of four consecutive blocks, and permutation $f_{i \bmod 4}$ is applied to the $i$-th block of blocks.
This pattern is continued with block sizes increasing by a factor of $4$ until $n$ is reached or exceeded.

\begin{figure}[!tb]
	\begin{subfigure}[b]{.49\linewidth}
		\includegraphics[page=1,height=6cm]{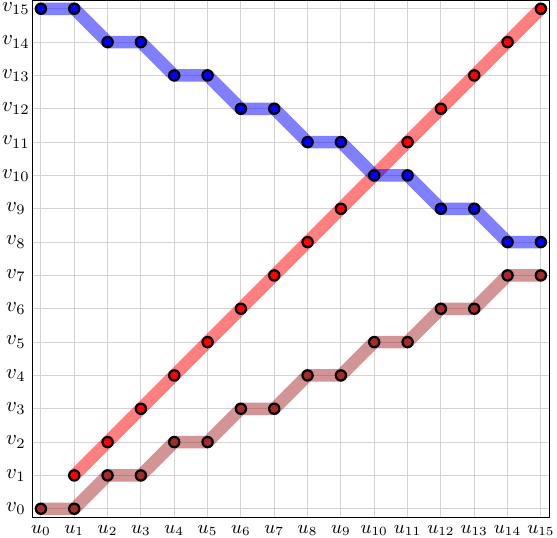}
		\subcaption{}
		\label{fig:g2-grid}
	\end{subfigure} \hfill
	\begin{subfigure}[b]{.49\linewidth}
		\includegraphics[page=2,height=6cm]{figures/graph2}
		\subcaption{}
		\label{fig:graph2_queues}
	\end{subfigure}
	\caption{$1$-stack $2$-queue separated layout of $G_{16}$ \textbf{(a)} and $4$-queue separated layout of $G_{16}$ \textbf{(b)}}
	\label{fig:graph2}
\end{figure}

We conclude that even for low-degree graphs with $\smn(G) = 3$, it is not obvious how to transform a separated mixed layout into a separated pure layout, say with few queues.

\medskip

While we do not know how to generalize~\cref{thm:1s1q} even just to separated $1$-stack $2$-queue layouts, we can narrow down the difficult case of \cref{open:rel3}.
If~\cref{open:rel3} is answered positively, we can find an $f(s,q)$-queue layout for any bipartite graph admitting a separated $s$-stack $q$-queue layout.
In~\cref{thm:same_order} we show three equivalent formulations of this problem.
First, observe that in the previous challenging example, the transformation from the separated mixed to the separated pure layout applies the same column and row permutation.
Now,~\cref{thm:same_order} states that we may always assume that rows and columns are identically permuted when transforming a separated mixed layout into a separated pure layout.

Further, we show in~\cref{lm:sep-3} that every separated $s$-stack $q$-queue layout can be transformed into a separated $1$-stack $f(s,q)$-queue layout, while this transformation does not change the separated queue number by too much.
This implies that one can restrict the attention to separated $1$-stack $q$-queue
layouts, for solving \cref{open:rel3}.

\SeparatedEquivalent*

\begin{proof}
    Observe that \eqref{equ:sep-1} immediately implies~\eqref{equ:sep-3}, which implies~\eqref{equ:sep-4}.
    That~\eqref{equ:sep-2} implies~\eqref{equ:sep-1} is also immediate, when adding isolated vertices to guarantee $|A|=|B|$.
    The proofs of the other directions are deferred, in particular see~\cref{lm:sep-2} for \eqref{equ:sep-1}~$\Rightarrow$~\eqref{equ:sep-2} and \cref{lm:sep-3} for \eqref{equ:sep-3}~$\Rightarrow$~\eqref{equ:sep-1}.
    For \eqref{equ:sep-4}, we show equivalence by proving that \eqref{equ:sep-4} implies \eqref{equ:sep-3} in~\cref{lm:mn6}.
 \end{proof}

\begin{lemma}
    \label{lm:sep-2}
    If there is a function $f$ such that every separated $s$-stack $q$-queue graph admits a separated $f(s,q)$-queue layout, then every separated $s$-stack $q$-queue layout of a bipartite graph $G = (A\cup B,E)$ with $|A|=|B|$ can be transformed into a separated $2f(s,q+1)^2$-queue layout by applying the same permutation to $A$ and $B$.
\end{lemma}

In other words, \cref{lm:sep-2} proves that \eqref{equ:sep-1} implies \eqref{equ:sep-2} in \cref{thm:same_order}.

\begin{proof}
    Let $u_1,\ldots,u_n$ and $v_1,\ldots,v_n$ be the order of $A$ and $B$ in the given separated $s$-stack $q$-queue layout of $G$.
    We add the ``identity'' queue $Q$, where for each $i$, $1 \leq i \leq n$, there is the edge $(u_i,v_i)$ (if not already present in $G$).
    Let $G' = (A \cup B, E \cup Q)$ be the resulting bipartite graph.
    By assumption, there exists a separated $f(s,q+1)$-queue layout of $G'$ with vertex order $\sigma_0$.
    We are looking for another vertex order $\sigma$ such that, compared to the initial order, $A$ and $B$ are permuted the same, that is, the columns and rows in the reduced adjacency matrix are identically permuted.
    Since we added the identity queue $Q$ to obtain $G'$ in the beginning, this is equivalent to restoring $Q$ to the diagonal in the matrix; see~\cref{fig:thm1}.

    \begin{figure}[!tb]
	\begin{subfigure}[b]{.28\linewidth}
		\center
		\includegraphics[page=1,width=\linewidth]{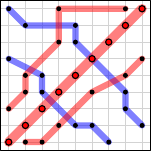}
		\subcaption{}
		\label{fig:thm1-1}
	\end{subfigure}
	\hfill
	\begin{subfigure}[b]{.28\linewidth}
		\center
		\includegraphics[page=2,width=\linewidth]{figures/identical-permutation}
		\subcaption{}
		\label{fig:thm1-2}
	\end{subfigure}
	\hfill
	\begin{subfigure}[b]{.28\linewidth}
		\center
		\includegraphics[page=3,width=\linewidth]{figures/identical-permutation}
		\subcaption{}
		\label{fig:thm1-3}
	\end{subfigure}
        \caption{
            Illustration of the proof of \cref{lm:sep-2}:
            Adding the identity queue \textbf{(a)}, obtaining a separated queue layout \textbf{(b)}, and restoring the identity queue \textbf{(c)}.}
        \label{fig:thm1}
    \end{figure}

    To that end, we shall apply~\cref{lm:321} to change only the order of vertices in $A$, which corresponds to permuting the columns in the matrix.
    Let $k = f(s,q+1)$ and $E_1,\dots, E_k$ be the queues in the $f(s,q+1)$-queue layout of $G'$ and $A_i = \{u \in A \mid (u,v) \in Q  \cap E_i\}$, $1 \leq i \leq k$.
    We want to apply the Riffle-Lemma (\cref{lm:321}-\eqref{item:riffle-bipartite}) with the partition of $V$ into $A_1,\ldots,A_k$ and $B$.
    To this end, we keep the order of $B$ (the rows), and restore $Q$ as a diagonal by simply sorting the columns according to their row in $Q$.
    Since each $E_i$ is a queue, the relative order within each $A_i$ is kept intact, as required by~\cref{lm:321}.
    Hence, by~\cref{lm:321}-\eqref{item:riffle-bipartite}, the resulting layout has at most $2 \cdot f(s,q+1) \cdot f(s,q+1)$ queues in total.\qedhere
\end{proof}

\begin{lemma} \label{lm:sep-3}
    Let $G$ be a bipartite graph with a separated $s$-stack $q$-queue layout.
    Then there exists another bipartite graph $H$ admitting a separated $1$-stack $(s+q-1)$-queue layout, where $G$ is a $1$-shallow minor of $H$, such that $\qn(G) \in \Oh(\qn(H)^3)$. 
\end{lemma}

If (3) of \cref{thm:same_order} holds, then the separated $1$-stack $(s+q-1)$-queue layout of $H$ has bounded queue number, i.e, $\qn(G) \in \Oh(\qn(H)^3) \leq f(s+q-1)$ for some function $f$. Hence $\qn(G) \leq f'(s,q)=f(s+q-1)$, implying (1) of~\cref{thm:same_order}.

\begin{proof}
    Assume we have a bipartite graph $G$, that admits a separated $s$-stack $q$-queue layout.
    We will build a new bipartite graph $H$, such that $H$ has a separated $1$-stack $(s + q - 1)$-queue layout.
    We will then show that for the separated queue numbers of $G$ and $H$, it holds that $\qn(G) \in \Oh(\qn(H)^3)$.
    That is, a small queue number of $H$ implies a small queue number of~$G$.
    Next we make the argument formal.

    Let $G=(A \cup B, E), |A| = n, |B| = m$, $S_1,\ldots,S_s$ be the stacks and $Q_1,\ldots,Q_q$ be the queues in the separated $s$-stack $q$-queue layout.
    We build $H$ as follows.
    The vertices of $H$ are $A  \cup (\bigcup_k U^k) \cup B \cup (\bigcup_k V^k)$, where $A = \{a_1, \dots, a_n\}$ and $B = \{b_1, \dots, b_m\}$ are copied from $G$, while $U^k \subseteq \{u^k_1, \dots, u^k_n\}$ and $V^k \subseteq \{v^k_1, \dots, v^k_m\}$ for any $1 \leq k \leq s-1$ are new. $U^k$ contains a vertex $u_i^k$ for each vertex $a_i$ incident to an edge $(a_i,b_j) \in S_k$. Likewise, $V^k$ contains a vertex $v_j^k$ for each such $b_j$. Vertices in $A$ and $B$ that are not incident to an edge in $S_k$ are omitted in $U^k$ and $V^k$.
    (Note that for each stack $S_k$, the new vertices in $U^k$ and $V^k$ are added to construct parts of the graph $H$ starting from $G$, while the last stack $S_s$ is left as is.)
    In the reduced adjacency matrix of $H$, the new vertices are represented by additional blocks of $n_k$ consecutive rows for each $U^k$ and $m_k$ consecutive columns for each $V^k$, where $n_k$ ($m_k$) is the number of vertices of $A$ ($B$) having an edge in $S_k$. Row blocks of $U^k$ and column blocks  of $V_k$ are subsequent for $k = 1,\ldots,s-1$.
    Intuitively, we move the stack $S_k$ from $A \times B$ to $U^k \times V^k$, $k=1,\ldots,s-1$.
    We leave the last stack $S_s$ with all $q$ initial queues in $A \times B$.
    Finally, we put one new queue each in $A \times U^k$ and $V^k \times B$, $k = 1,\ldots,s-1$.
    \cref{fig:grid_sq} shows the construction and in particular the order of the rows $U^1,\ldots,U^{s-1}$ and columns $V^1,\ldots,V^{s-1}$.

     \begin{figure}[!tb]
        \begin{subfigure}[b]{.4\linewidth}
            \begin{subfigure}[t]{\linewidth}
    		\center
	       	\includegraphics[page=1]{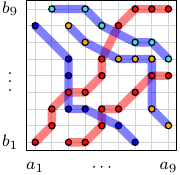}
		      \subcaption{}
            \end{subfigure}
            \bigskip
            \begin{subfigure}[b]{\linewidth}
		      \center
		      \includegraphics[page=4]{figures/1-s-k-q}
		      \subcaption{}
            \end{subfigure}
        \end{subfigure}
	\hfill
	\begin{subfigure}[b]{.59\linewidth}
		\center
		\includegraphics[page=3,width=\linewidth]{figures/1-s-k-q}
		\subcaption{}
	\end{subfigure}
	\caption{
            Transforming a separated $3$-stack $2$-queue layout into a separated $1$-stack $4$-queue layout: The initial graph $G$ \textbf{(a)}, modifying a stack edge~$(a_i, b_j)$ \textbf{(b)}, the constructed graph $H$~\textbf{(c)}.
        }
	\label{fig:grid_sq}
    \end{figure}

    More formally, for every edge $(a_i, b_j) \in Q_1 \cup \ldots \cup Q_q$, there is an edge $(a_i, b_j)$ in $H$.
    For every stack edge, $(a_i, b_j)$ of stack $S_k, 1\leq k \leq s-1$, there are three edges $(a_i, u^k_i), (u^k_i, v^k_j), (v^k_j, b_j)$ in $H$.
    (In case of multiple edges between a pair of vertices, we keep only one instance.)
    It is easy to verify that $H$ admits a separated $1$-stack $(s+q-1)$-queue layout, as for every original stack $S_k$, $1 \leq k \leq s-1$ the edges $(a_i, u^k_i)$ and $(v^k_j, b_j)$ for every $(a_i, b_j) \in S_k$ form a queue respectively, due to the increasing column and row indices of each block $U^k$ and $V^K$ starting from the bottom left of the matrix. In fact, edges between vertices in $A$ and $U^k$ and edges between vertices in $B$ and $V^k$ fit in a single queue, since the former lie below and to the left of the latter in the matrix. The original queues remain, yielding $q + s - 1$ queues in total. For a stack $S_k$, the edges $(a_i,b_j) \in S_k$ are a decreasing subset in $A \times B$ and due to the increasing ordering of indices in every column and row block of the matrix, the edges $(u^k_i,v^k_j)$ in $U^k \times V^k$ have the same property. The ordering of the row blocks $U^k$ and column blocks $V^k$, where $A \times B$ and $U^k \times V^k$, $1 \leq k \leq s-1$ are positioned diagonally from the top left to the bottom right of the matrix, then allows to combine the $s$ stacks of $G$ into one single stack in $H$.

    Finally, we show that $\qn(G) \in \Oh(\qn(H)^3)$.
    To this end, consider a $\qn(H)$-queue layout of $H$.
    Now, contract for each $i = 1,\ldots,n$ the vertex $a_i$ with its neighbors of the form $u^k_i$, $1 \leq k \leq s-1$.
    Similarly, contract for each $j = 1,\ldots,m$ the vertex $b_j$ with its neighbors of the form $v^k_j$, $1 \leq k \leq s-1$.
    The result is a $1$-shallow minor of $H$, and most crucially, the result is exactly the initial graph $G$.
    Hence, \cref{lm:contraction} with radius $r=1$ gives that $\qn(G) \le (2r+1)(2 \qn(H))^{2r+1} = 24 \cdot \qn(H)^{3}$.
\end{proof}

Using the techniques from \cref{sec:non-separated-layouts}, we can provide the final missing piece of \cref{thm:same_order}.

\begin{lemma}
    \label{lm:mn6}
    Let $G$ be a bipartite graph with a separated $1$-stack $q$-queue layout.
    Then $G$ has a subdivision $D$ with	at most $2 \lceil \log_2 q \rceil$ division vertices per edge that admits a separated $1$-stack $6$-queue layout.
\end{lemma}

In other words, \cref{lm:mn6} proves that \eqref{equ:sep-4} implies \eqref{equ:sep-3} in \cref{thm:same_order}, as applying~\cref{lm:contraction} for the r-shallow minor $G$ of $D$ with $k = 2 \lceil \log_2 q \rceil$ and $r = \lceil (k+1)/2 \rceil$ implies bounded queue number of $G$ for a function $f(q)$.
\[
		\qn(G) \le 	(2r + 1) \big(2 \qn(D) \big)^{2r + 1}
		= (2 \lceil \log_2 q \rceil + 2) \big(12)^{2 \lceil \log_2 q \rceil + 2}
	\]
\begin{proof}
    Let $G = (V, E_s \cup E_1 \cup \dots \cup E_q)$ be a bipartite graph that admits a separated $1$-stack $q$-queue layout such that $E_s$ is a stack and $E_i, 1 \le i \le q$ are queues.
    Denote $h = \lceil \log_2 q \rceil$ so that $q \le 2^h$.
    We consider the subgraph of $G$ induced by the queue edges, $G^q=(V, E_1 \cup \dots \cup E_q)$.
    Let $T$ be a complete binary tree of height $h$.
    By \cref{lm:t_subdiv}, there exists a subdivision of $G^q$, denoted $D^q$, and a simple $T$-layout of $D^q$ in which $\Sh(x) = 0, \Qh(x) = 1$ for all $x \in \Vleaf(T)$, $\Sh(x) = \Qh(x) = 0$ for all non-leaf nodes $x \in V(T) \setminus \Vleaf(T)$,	and $\Kh(x, y)=1$ for all $(x, y) \in E(T)$.
    Note that $D^q$ contains exactly $2h$ division vertices per edge and it is bipartite; see \cref{fig:mn6}.

    \begin{figure}[!tb]
	\begin{subfigure}[b]{.56\linewidth}
		\center
		\includegraphics[page=5,width=\linewidth]{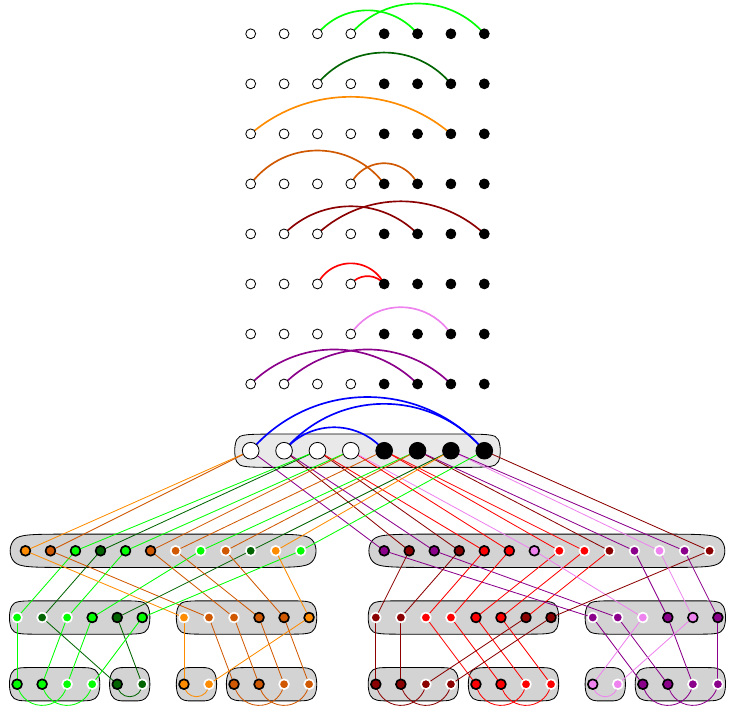}
		\subcaption{A tree-layout of the subdivision of a separated $1$-stack $8$-queue graph $G$. Colors of edges and division vertices in the tree-layout correspond to the original queues and stack of $G$ shown above the root node, while vertex outline colors in non-root nodes correspond to the vertex partitions in G.}
	\end{subfigure}
	\hfill
	\begin{subfigure}[b]{.42\linewidth}
		\center
		\includegraphics[page=6,height=5cm]{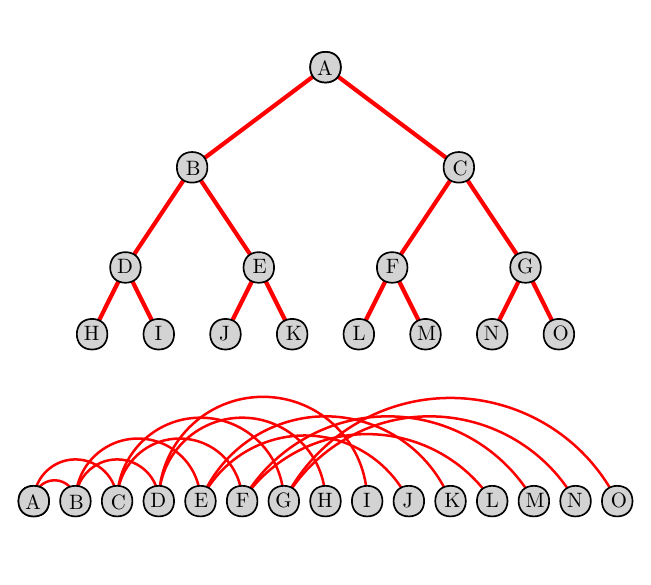}
		\subcaption{\nolinenumbers{}A tree-partition and a layout of the corresponding binary tree $T$, where the red edge color corresponds to~\cref{lm:t_layout}.}
	\end{subfigure}
	\caption{An illustration for \cref{lm:mn6}:
		Subdividing a separated $1$-stack $8$-queue graph $G$
		with at most $2 \lceil \log_2 q \rceil=6$ division vertices per edge
		to obtain a mixed $1$-stack $6$-queue graph.}
	\label{fig:mn6}
    \end{figure}

    Now, color all edges of $T$ red and find a $1$-queue layout of $T$ via a breadth-first search traversal such that every node precedes its children in the order.
    By \cref{lm:t_layout} applied to graph $G^q$ and its simple $T$-layout, we obtain a linear layout of $D^q$.
    Let us argue that this result is in fact a $3$-queue layout of $D^q$, i.e., that $\lambda_s = 0$ and $\lambda_q \le 3$ as defined in \cref{lm:t_layout}.
    On the one hand, $\Sh(x) = 0$ for all the nodes of $T$ and there are no blue edges in $T$.
    Thus, we have $\lambda_s=0$.
    On the other hand, every node $x \in V(T)$ has at most two outgoing edges in the layout of $T$, i.e., the set $\{ (y,z) \in E^q \mid y \le_{\sigma} x \le_{\sigma} z \}$ contains at most two edges.
    Hence,
    \begin{equation*}
	\begin{aligned}
		\lambda_q & = \max_{x \in V(T)} \Big( \Qh(x) +  \max_{y \in V(T) \colon y \le_{\sigma} x}  \sum_{(y,z) \in E^q \colon x \le_{\sigma} z} \Kh(y, z)  \Big) \\
			& \le \max_{x \in V(T)} \Big( 1 +  \max_{y \in V(T) \colon y \le_{\sigma} x}  2  \Big) \le 3.
	\end{aligned}
    \end{equation*}

    The final step is to convert the $3$-queue layout of $D^q$ into a separated layout.
    This is accomplished by	\cref{cor:separation}, which in worst case doubles the number of queues.
    The transformations of \cref{lm:t_layout} and \cref{cor:separation} keep the relative order of the original vertices $V$ unchanged.
    Thus, the resulting layout of $D^q$ can be joined with the stack edges $E_s$ to finally yield a subdivision $D$ of $G$ (in which the stack edges are not subdivided) with a separated $1$-stack $6$-queue layout.
\end{proof}

\section{Conclusions and Open Questions}
\label{sec:conclusions}

We have explored the relation between mixed and pure stack or queue layouts, in the separated as well as in the non-separated setting.
An interesting (and somewhat surprising) result is the equivalence of \cref{open:rel1,open:rel2}, which we believe sheds some light on the famous open problem, whether bounded stack number always implies bounded queue number.
Let us highlight a few intriguing questions and possible directions in the area.

\begin{enumerate}
    \item Do graphs with (non-separated) $1$-stack $1$-queue layouts have a constant queue number?
        This might be hard in general, but one could for example start with subcubic graphs.

    \item Do graphs with separated $1$-stack $2$-queue layouts have a constant queue number?
        \cref{lm:sep-2,lm:mn6} combined show that separated $1$-stack $6$-queue graphs are as hard as the general case, and we feel that the same holds for separated $1$-stack $2$-queue graphs.
        On the other hand, \cref{thm:1s1q} solves the separated $1$-stack $1$-queue case.

    \item Is there a constant $C$ such that every bipartite $1$-stack $1$-queue graph admits a separated $C$-stack $C$-queue layout?
        A positive answer would imply that \cref{open:rel3,open:rel1,open:rel2} are all equivalent.
\end{enumerate}

\bibliographystyle{plainurl}
\bibliography{references}

\end{document}